\documentclass[11pt]{amsart}

\usepackage{amssymb,epsf,psfrag}
\usepackage{graphicx}
\usepackage{color}
\usepackage{amsfonts}
\usepackage{enumerate}
\usepackage{latexsym}
\usepackage{epsfig}
\usepackage{amsthm}
\usepackage{amsmath}
\usepackage{latexsym}
\usepackage[all]{xy}

\usepackage[
            breaklinks=true,
            colorlinks=true]{hyperref}
\usepackage[active]{srcltx}

\makeatletter\@addtoreset{equation}{section}\makeatother
\makeatletter\@addtoreset{figure}{section}\makeatother
\makeatletter\@addtoreset{table}{section}\makeatother

\makeatletter\@addtoreset{equation}{section}\makeatother
\makeatletter\@addtoreset{figure}{section}\makeatother
\makeatletter\@addtoreset{table}{section}\makeatother

\newtheorem{theorem}{Theorem}[section]
\newtheorem{prop}[theorem]{Proposition}

\newtheorem{problem}[theorem]{Problem}

\newtheorem{cor}[theorem]{Corollary}

\newtheorem{theoremL}{Theorem}
\newtheorem{conjL}{Conjecture}
\newtheorem{questionL}{Question}

\newcommand{\R}{{\mathbb R}}
\newcommand{\Z}{{\mathbb Z}}

\newcommand{\N}{{\mathbb N}}

\newcommand{\op}[1]{\!\!\mathop{\rm ~#1}\nolimits}

\newenvironment{remark}{\refstepcounter{theorem}\par\medskip\noindent{\bf
Remark~\thetheorem~~}}{\unskip\nobreak\hfill\hbox{ $\oslash$}\par\bigskip}

\newenvironment{question}{\refstepcounter{theorem}\par\medskip\noindent{\bf
Question~\thetheorem~~}}{\unskip\nobreak\hfill\hbox{ $\oslash$}\par\bigskip}

\newenvironment{example}{\refstepcounter{theorem}\par\medskip\noindent{\bf
Example~\thetheorem~~}}{\unskip\nobreak\hfill\hbox{ $\oslash$}\par\bigskip}

\newenvironment{definition}{\refstepcounter{theorem}\par\medskip\noindent{\bf
Definition~\thetheorem~~}}

\renewcommand{\geq}{\geqslant}
\renewcommand{\leq}{\leqslant}

\begin{document}

\title[Optimization problem for 
non\--Hamiltonian periodic  flows]{An integer optimization problem \\  
for non\--Hamiltonian periodic flows}

\author{\'Alvaro Pelayo \,\,\,\,\,\,\,\,\,\,\,\, \,\,\,\,\,\,\,Silvia Sabatini}
\date{}

\maketitle

\begin{abstract}
Let $\mathcal{C}$ be the class of compact $2n$\--dimensional symplectic manifolds $(M,\omega)$
for which the first or $(n-1)$ Chern class vanish. We point out an integer 
optimization problem to find a lower bound $\mathcal{B}(n)$ 
on the number of equilibrium points of non\--Hamiltonian symplectic periodic flows on 
manifolds $(M,\omega) \in \mathcal{C}$. 
As a consequence, we confirm in dimensions $2n \in \{8,10,12,14,18,20, 22\}$
a conjecture for unitary manifolds made by  Kosniowski in 1979  
for the subclass $\mathcal{C}$. 
\end{abstract}

\section{Introduction}

We point out how to pose a classical geometric problem 
concerning the size of  fixed point sets of  $S^1$\--actions on $2n$\--dimensional manifolds, 
as an integer programming problem.  This 
classical problem goes back to Frankel (1959) and Kosniowski (1979). 
By studying the integer programming problem we obtain some information 
on the classical problem for low values of $n$, under certain topological
assumptions.

\subsection{A geometric problem}
A \emph{symplectic form} on a smooth manifold $M$ is a closed, non\--degenerate
two\--form $\omega \in \Omega^2(M)$. \emph{For simplicity, throughout we assume that $M$
is connected}. Let $S^1 \subset \mathbb{C}$ be the multiplicative group 
of unit complex numbers.  
We say that an $S^1$\--action on a symplectic manifold $(M,\omega)$ 
is \emph{symplectic} if it preserves $\omega$.  Let 
$\mathcal{X}_M$ be the vector field on $M$ induced by the $S^1$ action. 
An $S^1$\--action is \emph{Hamiltonian} if the $1$\--form
$$\iota_{\mathcal{X}_M} \omega:=\omega(\mathcal{X}_M,\cdot)$$ is exact, that is,  there exists 
 a smooth map $\mu \colon M \to \R$ such that
$- \op{d}\! \mu = \iota_{\mathcal{X}_M} \omega$. 
The map $\mu$ is called a \emph{momentum map}. 
In the terminology of dynamical systems, we refer to symplectic $S^1$\--actions as
\emph{symplectic periodic flows}. The fixed points of the action correspond to the equilibrium points
of the flow. 

A classical question in symplectic
geometry is whether symplectic $S^1$\--actions on compact symplectic manifolds which possess 
nonempty discrete fixed point set are Hamiltonian. It is unknown whether there is a non\--Hamiltonian
 symplectic $S^1$ action on  a compact symplectic manifold $M$ for which the 
 fixed set $M^{S^1}$ is discrete but not empty.  If the point set is not required to be discrete,  
there are non\--Hamiltonian symplectic $S^1$\--actions with no fixed points; for
instance,  given an even\--dimensional torus, consider the natural action of any
circle subgroup.  In \cite[Proposition~1 and Section~3]{McDuff1988}, McDuff  constructs  
a non\--Hamiltonian symplectic $S^1$\--action on a six\--dimensional compact
symplectic manifold,  whose fixed sets are tori.  If the fixed point set is discrete and nonempty, the problem
has a  history of partial answers, but the solution is unknown in general.  
The following question is closely related to this problem.

\begin{questionL} \label{qst}
\emph{Given a compact symplectic $2n$\--dimensional manifold $(M,\omega)$ with a symplectic $S^1$\--action 
and nonempty discrete fixed point set $M^{S^1}$, what 
are lower bounds $\mathcal{B}(n)$ on the cardinality of $M^{S^1}$  
depending on the $n$}?
\end{questionL}

Question~\ref{qst} goes back to  Frankel (1959), who gave a sharp bound for
compact K\"ahler manifolds. A conjecture in this direction (Conjecture~\ref{K1979}) 
was made for the larger class of unitary
manifolds by Kosniowski (1979), which remains open in general, as far as we know. The work of
McDuff in 1988 (Theorem~\ref{dusa}) gives a sharp answer for compact symplectic $4$\--manifolds.
We review these results in Section~\ref{s2}. 

\subsection{Summary of results}
In Section~\ref{s3} we prove a result
(Theorem~\ref{ab}), the proof of which uses \cite[Theorem 1.2]{GoSa12} to give a non\--explicit formula for $\mathcal{B}(n)$, 
provided that the first or $(n-1)$ Chern class of $M$ vanish.
Theorem~\ref{ab} gives rise to an integer optimization problem (Problem~\ref{op}, see also 
Problem~\ref{op2})  to find $\mathcal{B}(n)$. The 
problem may be solved by hand for several values of $n$. 
In this way we confirm a conjecture of Kosniowski of 1979 in dimensions 
$8\leq 2n \leq 22$, $2n \neq 16$ for the class of symplectic manifolds with
vanishing first or $(n-1)$ Chern class.

By studying Theorem~\ref{ab} we 
obtain some properties related to the fixed points of the $S^1$\--action, namely,
we give necessary conditions on the number of negative weights at the fixed points (e.g. 
Theorem \ref{est1}).  As a consequence, when 
$n=2m+1$ and $m\notin\{6k(k+1)+1\mid k\in \Z_{\geq 0}\}$, we show
that $\mathcal{B}(n) \geq 4$. The estimates are in some cases  higher than
known estimates: for instance,  $\mathcal{B}(5)=24$ which improves the known estimate of $3$. 
 However, we have not solved the integer optimization problem we introduce to 
give a formula for $\mathcal{B}(n)$ for general $n\in \mathbb{N}$. We suspect that it would in some cases be amenable to techniques
of linear programming, and if so, it could possibly lead to the solution of the conjecture in more cases.

\section{Preliminares} \label{s2}

We review some  results which are relevant  to the answers given to
Question~\ref{qst} in this article, including some which we will need in the proofs.

\subsection{Origins}

At least from the point of view of equivariant symplectic geometry of torus actions,
 the interest in Question~\ref{qst} 
has been to a large extent motivated by a result by T. Frankel on $S^1$\--actions 
on compact K\"ahler manifolds.

\begin{theorem}[Frankel '59]\label{general0}
A symplectic K\"ahler $S^1$\--action on a compact K\"ahler manifold is Hamiltonian
if and only if it has fixed points. 
\end{theorem}

The following appeared in McDuff  \cite[Proposition~2]{McDuff1988}.
 
 \begin{theorem}[McDuff, '88] \label{dusa}
 A symplectic $S^1$\--action on a $4$\--dimensional
compact and connected symplectic manifold with at least one fixed point
is Hamiltonian. 
\end{theorem}

On the other hand, the following is a classical fact.

\begin{prop} \label{easy}
  A Hamiltonian $S^1$\--action on a compact 
symplectic $2n$\--manifold  has at least 
$n + 1$
fixed points. 
\end{prop}

Proposition~\ref{easy} follows from the fact the fixed points of
the $S^1$\--action are in correspondence with the critical
points of the the momentum map
$\mu$ of the action, and $\mu$ is a perfect Morse function which satisfies the Morse
inequalities; we provide a detailed proof in Section~\ref{fr}.  
In view of Proposition~\ref{easy}, Theorem~\ref{general0}
gives the following.

\begin{cor}[Frankel '59]\label{general2}
A symplectic K\"ahler $S^1$\--action on a compact and connected K\"ahler $2n$\--manifold
has at least $n+1$ fixed points.
\end{cor}

We recall that a manifold $M$ is called \emph{unitary} if the stable tangent bundle is endowed with a complex
structure. Thus, 
every almost complex manifold, and hence every symplectic manifold, is unitary.
For a real number $r$, let $[r]$ denotes its integer part. The following can be found in \cite{Ko}.

\begin{conjL}[Kosniowski '79] \label{K1979}
Let $M$  be a  $2n$\--dimensional unitary $S^1$\--manifold with isolated fixed points. If 
 $M$ does not bound equivariantly, then the number of fixed points is at least $f(n)$, where
 $f(n)$ is a linear function, expected to be equal to $[n/2]+1$. 
\end{conjL}

\subsection{Some recent contributions}
Let $S^1$ act on a compact symplectic manifold $(M,\, \omega)$
with momentum map $\mu \colon M \to \R$.
Because the set of compatible almost complex structures
$J \colon \op{T}\!M \to \op{T}\!M$ is contractible,  there is a well\--defined
total Chern class of the tangent bundle ${\rm T}M$, which we denote by
${\rm c}^M=\sum_{j=0}^n{\rm c}_j^M.$ For every fixed point $p$ there is a
well\--defined 
multiset of integers, namely the \emph{weights} of the $S^1$ action on ${\rm T}M|_{p}$.  Let 
 $\op{c}_1(M)|_p$ be the first equivariant Chern class of ${\rm T}M$
 at  $p \in M^{S^1}$, which one can naturally identify with an integer 
 $\op{c}_1(M)(p)$: the sum
 of the weights at $p$.
The map
 $$
 \op{c}_1(M) \colon M^{S^1} \to \Z, \,\,\,\,\,\,\,p \mapsto \op{c}_1(M)(p) \in \Z,
 $$
 is called 
 the \emph{Chern class map} of $M$. Using the Atiyah\--Bott and Berline\--Vergne localization formula in equivariant
cohomology and  properties of Poincar\'e polynomials and Morse functions,
Pelayo and Tolman (\cite{PeTo2010}) proved the following result.

\begin{theorem}[\cite{PeTo2010}]\label{general}
Let $S^1$  act symplectically
on a compact  symplectic $2n$\--dimensional manifold $M$.
If ${\rm c}_1(M) \colon M^{S^1}\to \Z$ is somewhere injective\footnote{Let $f\colon X \to Y$ be 
a map between sets. We recall that
$f$ is \emph{somewhere injective}
if 
there is  $y\in Y$ such that $f^{-1}(\{y\})$
is the singleton.}  then
 the $S^1$\--action has at least  $n + 1$ fixed points.
\end{theorem}

Prior to Theorem \ref{general}, Tolman and Weitsman had proven a 
remarkable result for semifree actions.

\begin{theorem}[\cite{tolmanweitsman}] \label{TTS}
Let $S^1$  act symplectically and semifreely 
on a compact  symplectic $2n$\--dimensional manifold $M$  with isolated fixed points.
Then the $S^1$\--action has at least  $n + 1$ fixed points.
\end{theorem}

The following result complements Theorems~\ref{general} and \ref{TTS} in low dimensions.

\begin{theorem}[\cite{PeTo2010}]\label{maintheorem}
Let $S^1$ act symplectically
on a compact  symplectic manifold $M$. 
Suppose that the fixed point set $M^{S^1}$
is nonempty.
Then
there are at least  two fixed points. If
$\dim M\ge 8$, 
then  there are at least three fixed points.
Moreover, if ${\rm c}_1(M) \colon M^{S^1}\to \Z$ is not identically zero and 
$\dim M\ge 6$, 
then there are at least four fixed points.
\end{theorem}

The literature on non\--necessarily Hamiltonian $S^1$\--symplectic actions 
is extensive, see \cite{tolmanweitsman, Feldman2001, 
Lin2007, CHS10, Lin2007, LP12, go, Go2006, LL11, LiLi2010} and the 
references therein for related results. Non\--necessarily Hamiltonian $(S^1)^k$\--actions
on manifolds of dimension $2k$ where studied in \cite{DP07,Pe10}.

\section{An  integer programming problem} \label{s3}

\subsection{Tools}

In order to state and prove our results we need the following definition
and the two previous results which follow it.

\begin{definition}
Let $(M,\omega)$ be a compact symplectic manifold on which a circle
$S^1$ acts symplectically with nonempty discrete fixed point set $M^{S^1}$.  For $i \in \Z$ let 
$$
N_i:=\Big\lvert \big\{ p \in M^{S^1} \ \big| \ \lambda_p = i \big\} \Big\rvert,
$$
where  $\lambda_p$ is the number of negative weights at  $p$
for all $p \in M^{S^1}$.
\end{definition}

The following was proven in Pelayo\--Tolman~\cite{PeTo2010}.

\begin{theorem}[\cite{PeTo2010}] \label{peto}
Let $S^1$ act symplectically
on a compact symplectic $2n$\--manifold with isolated fixed points.
Then
\begin{eqnarray*} 
N_i=N_{n-i} \quad 
\forall \ i \in \Z.
\end{eqnarray*}
\end{theorem}

The following corresponds to \cite[Theorem~1.2]{GoSa12}.

\begin{theorem}[\cite{GoSa12}]\label{GS}
Let $(M,J)$ be an almost complex compact and connected manifold equipped with an $S^1$\--action which preserves the almost complex structure $J$ and has a discrete fixed point set. For every $p = 0,\ldots, n$, let $N_p$ be the number of fixed points with exactly $p$ negative weights. Then
$$
\int_M {\rm c}_1^M {\rm c}_{n-1}^M=\sum_{p=0}^n N_p\Big(6p(p-1)+\frac{5n-3n^2}{2}  \Big),
$$
where ${\rm c}_1^M$ and ${\rm c}_{n-1}^M$ are respectively the first and
$(n-1)$ Chern classes of $M$.
\end{theorem}

Because of Theorem~\ref{GS}, the main equation for  the purposes of this article is:
\begin{eqnarray} \label{assum}
\int_M {\rm c}_1^M {\rm c}_{n-1}^M=0.
\end{eqnarray}

\subsection{Discussion on assumption (\ref{assum})}

We start with a remark.

\begin{remark} \label{fg}
The following hold:
\begin{itemize}
\item[(i)]
if  ${\rm c}^M_1=0$ or ${\rm c}^M_{n-1=0}$ then (\ref{assum}) holds;
\item[(ii)]
if a symplectic $S^1$\--action is Hamiltonian then ${\rm c}^M_1 \neq 0$ (see for example \cite[Lemma 3.8]{T10}).
\end{itemize}
\end{remark}

Condition \eqref{assum} is the  assumption under which the results of this paper are proven.
Next we show that in dimension 6 \eqref{assum} is equivalent to requiring the $S^1$\--action to be 
non\--Hamiltonian. First we recall a few definitions and known facts.

By \cite{Au1991} (see also \cite[Proof of Theorem 1, paragraph 1]{Feldman2001}):
\begin{equation}\label{N0}
 \mbox{if the action is non-Hamiltonian then $N_0=N_n=0$.}
\end{equation}
Fact (\ref{N0}) was used by Feldman in \cite{Feldman2001} to characterize the Todd genus of a compact and connected symplectic manifold endowed with a symplectic $S^1$\--action
and discrete fixed point set.
We recall that, given a compact almost complex manifold $(M,J)$, the \emph{Todd genus} ${\rm Todd}(M)$ is the genus associated to the power series $$\displaystyle\frac{x}{1-e^{-x}}.$$ The following result is due to Feldman.

\begin{theorem}[\cite{Feldman2001}]\label{fel}
The Todd genus associated to a compact and connected symplectic manifold with a symplectic $S^1$\--action and discrete fixed point set is either one, in which case
the action is Hamiltonian, or zero, in which case the action is not Hamiltonian.
\end{theorem}

We are ready to prove the following.

\begin{prop} \label{cv}
Suppose that $S^1$ acts symplectically 
on a compact and connected $6$\--dimensional manifold
$M$ with nonempty discrete fixed point set. Then the $S^1$\--action is non\--Hamiltonian
if and only if \eqref{assum} holds.
\end{prop}

\begin{proof}
It is sufficient to observe that when $\dim(M)=6$, $${\rm Todd}(M)=\displaystyle\int_{M}\frac{{\rm c}_1^Mc^M_2}{24}.$$ 
The conclusion follows from Theorem \ref{fel}.
\end{proof}

\begin{question}
Under which conditions does the claim in Proposition~\ref{cv} hold if $2n\geq 8$?
\end{question}

\subsection{Integer programming problem}
In this section we present an approach for finding the minimal number of fixed points on a compact and connected symplectic manifold
satisfying \eqref{assum}
endowed with a symplectic, but non-Hamiltonian, $S^1$\--action with discrete fixed point set.
In virtue of Theorem \ref{dusa} we will henceforth assume that $$2n=\dim(M)\geq 6.$$

Let $m \in \Z_{\geq 0}$ and let 
\begin{eqnarray}
F_1(N_1,\ldots,N_m)&:=& N_m+ 2\, \sum_{k=1}^{m-1} N_{m-k}; \label{f1} \\
F_2(N_1,\ldots,N_m)&:=&  2\, \sum_{k=1}^{m} N_{k}; \label{f2} \\
G_1(N_1,\ldots,N_m)&:=& -mN_m+ 2\, \sum_{k=1}^{m-1} (6k^2-m) \, N_{m-k} ;\nonumber\\
G_2(N_1,\ldots,N_m)& := & \sum_{k=0}^{m-1}\Big(6k(k+1)-m+1\Big)N_{m-k}.\nonumber
\end{eqnarray} 
For $i \in \{1,\,2\}$ let
\begin{equation} \label{z}
\mathcal{Z}_i:=\Big\{N:=(N_1,\ldots,N_m) \in (\Z_{\geq 0})^m \,\,\, | \,\,\, F_1(N)>0,\,\,\,G_1(N)=0  \Big\}. 
\end{equation}
We are ready to state our main result.

\begin{theoremL} \label{ab}
Let $(M,\omega)$ be a $2n$-dimensional compact and connected symplectic manifold
with a symplectic but non-Hamiltonian $S^1$\--action with 
nonempty, discrete fixed point set and 
such that {\rm (\ref{assum})} holds. Let $F_1,\,F_2$ be 
respectively given in \eqref{f1},\, \eqref{f2}, and let $\mathcal{Z}_1,\mathcal{Z}_2$
be given in \eqref{z}.
Then the number of fixed points of the $S^1$\--action is greater than or equal to:
$$
\mathcal{B}(n):= 
\left\{ \begin{array}{rl}
   \min_{\mathcal{Z}_1}
   F_1 & \,\,\,\,\,\,\,\,\,\,\,\,\,\,\, {\rm if}\,\, \, n =2m;   \\
    \textup{\,} \\
    \min_{\mathcal{Z}_2}
 F_2  & \,\,\,\,\,\,\,\,\,\,\,\,\,\,\, {\rm if}\,\, \, n=2m+1.
  \end{array}   \right.
$$
\end{theoremL}

\begin{proof}
By Theorem~\ref{peto} we have that $N_i=N_{n-i}$ for every $i \in \Z$.  Since 
the $S^1$\--action
is symplectic but not Hamiltonian, by \eqref{N0} we have $N_0=N_n=0$. Thus, since the
 total number of fixed points is $$\sum_{k=1}^{n-1}N_k,$$ 
it follows that 
$F_1(N_1,\ldots,N_m)$ (resp.~$F_2(N_1,\ldots,N_m)$) counts the total number of fixed points 
when $n=2m$ (resp. $n=2m+1$), and, since the fixed point set is nonempty, 
we have $F_1>0$ (resp. $F_2>0$). 

Moreover the constraint $G_1=0$
(resp. $G_2=0$) comes from combining \eqref{assum} with Theorem \ref{GS}.
Let $g(p,n)$ be $$6p(p-1)+\displaystyle\frac{5n-3n^2}{2}.$$ 
\\ If $n=2m$, by Theorem~\ref{peto} and \eqref{N0}, a computation shows that
\begin{equation}\label{G1=0}
\begin{aligned}
\sum_{p=0}^nN_p\, g(p,n) & =   -mN_m+\sum_{k=1}^{m-1}  \Big(g(m-k,n)+g(m+k,n)\Big) \, N_{m-k} \\
 & =-mN_m+2\sum_{k=1}^{m-1} (6k^2-m) \, N_{m-k} \\
 &=G_1(N_1,\ldots,N_m). \nonumber
\end{aligned}
\end{equation}
Analogously, if $n=2m+1$ we have that
\begin{equation}\label{G2=0}
\begin{aligned}
\sum_{p=0}^nN_p\, g(p,n) & =  \sum_{k=0}^{m-1}N_{m-k} \Big( g(m-k,n)+g(m+k+1,n) \Big) \\
& =2\sum_{k=0}^{m-1}\Big(6k(k+1)-m+1\Big) \,N_{m-k} \\
&=2\, G_2(N_1,\ldots,N_m). \nonumber
\end{aligned}
\end{equation}
If $n=2m$ (resp. $n=2m+1$), by Theorem \ref{GS} the constraint $G_1=0$ (resp. $G_2=0$) is equivalent to \eqref{assum}.
\end{proof}

The following integer programming problem, which is motivated by the conjecture of
Kosniowski (Conjecture~\ref{K1979}), arises from Theorem~\ref{ab}.

\begin{problem} \label{op}
\emph{Let $i \in \{1,\,2\}$.  Let $F_1,\,F_2$ be 
respectively given in \eqref{f1},\, \eqref{f2}, and let $\mathcal{Z}_1,\mathcal{Z}_2$
be given in \eqref{z}.  Find conditions on $n \in \Z$, $n\geq 3$, such that
\begin{eqnarray} \label{xc}
\min_{\mathcal{Z}_i}   F_i \,\,\, \geq \,\,\, [n/2]+1
\end{eqnarray}
holds. }
\end{problem}

The table in Figure~\ref{abd} provides a solution to Problem~\ref{op} when
$2n \in \{8,10,12,14,18,20, 22\}$. We have not solved it for $2n \geq 26$.

Similarly, motivated by Frankel's Theorem (Theorem~\ref{general2}), we propose the following
sharper version of Problem~\ref{op}.

\begin{problem} \label{op2}
\emph{Let $i \in \{1,\,2\}$.  Let $F_1,\,F_2$ be 
respectively given in \eqref{f1},\, \eqref{f2}, and let $\mathcal{Z}_1,\mathcal{Z}_2$
be given in \eqref{z}.  Find conditions on $n \in \Z$, $n \geq 3$, such that
\begin{eqnarray} \label{xc2}
\min_{\mathcal{Z}_i}   F_i \,\,\, \geq \,\,\, n+1
\end{eqnarray}
holds. }
\end{problem}

\begin{example}\label{ex1}
Using Theorem~\ref{ab} we can compute $\mathcal{B}(n)$ for some values of $n$.
The table in Figure \ref{abd} gives $\mathcal{B}(n)$ for $n\leq 12$. For the sake of clarity we compute $\mathcal{B}(n)$ when $n=4$ and $5$, the other cases are analogous.
\begin{itemize}
\item $2n=8$. We have to minimize $F_1(N_1,N_2)=N_2+2N_1$ subject to the conditions $G_1(N_1,N_2)=-2N_2+8N_1=0$, $F_1>0$ and$N_1,N_2\in \Z_{\geq 0}$.
One immediately sees that the values of $N_1$ and $N_2$ which minimize $F_1$ are respectively $1$ and $4$, yielding $\mathcal{B}(4)=6$.
\item $2n=10$. In this case we have to minimize $F_2(N_1,N_2)=2(N_1+N_2)$ subject to the conditions  $G_2(N_1,N_2)=-N_2+11N_1=0$, $F_2>0$ and $N_1,N_2\in \Z_{\geq 0}$.
A computation shows that the values of $N_1$ and $N_2$ which minimize $F_2$ 
are respectively $N_1=1$ and $N_2=11$, yielding $\mathcal{B}(5)=24$.
\end{itemize}
\end{example}

\begin{figure}[h] \label{abd}
  \centering
  \begin{tabular}{cccl}
     $\dim M=2n$& \,\,\,\,\,$n+1$ & \,\,\,\,\, $[n/2] +1$  & minimal $|M^{S^1}|$ if $\int_M c^M_1 c^M_{n-1} =0$\\
    \hline
  {8} &5& 3&  6\\
 {10} &6& 3 & 24 \\
 12 &7& 4& 4 \\
   {14} &8&4& 12 \\
 {\bf 16} &{\bf 9}&{\bf 5}& {\bf 3} \\
   18 &10 &5& 8 \\
      20 &11 &6& 12 \\
      22 &12 &6& 6 \\
      {\bf 24} & {\bf 13} & {\bf 7} & {\bf 2} \\
 \vdots &\vdots && \vdots\\
    $2n$ &$n+1$ && $\mathcal{B}(n)$\\
     \vdots &&& \vdots
  \end{tabular}
 \caption{Minimal number of fixed points of a symplectic $S^1$\--action 
 with nonempty discrete fixed point
  set on a compact and connected $2n$\--dimensional manifold, 
  under assumption (\ref{assum}). 
  The $[n/2]+1$ bound appears in Conjecture~\ref{K1979}. 
  The $n+1$ bound is motivated by Theorem~\ref{general2}. Values in boldface are those
  for which $\mathcal{B}(n)$ is smaller than the value predicted in  Conjecture~\ref{K1979}.}
\end{figure}

\begin{remark} \label{xxy}
The proof of Theorem~\ref{GS} in \cite{GoSa12} makes use of
equivariant $K$\--theory. Theorem~\ref{GS} is used to
prove Theorem~\ref{ab}. Theorem~\ref{ab} leads to estimates for $\mathcal{B}(n)$ which, in some cases, 
 improve previous estimates (see Figure~\ref{abd}). On the other hand, \cite{PeTo2010} uses equivariant 
 cohomology  to estimate $\mathcal{B}(n)$
(under different assumptions), eg. see  Theorem~\ref{general}.
\end{remark}

We have not solved the optimization problem arising from Theorem~\ref{ab}  to
estimate $\mathcal{B}(n)$, $n \in \N$ (see also Problems~\ref{op}, \ref{op2}). 
However, we suspect that it could be amenable to techniques
from linear programming. If this is the case, Theorem~\ref{ab} 
could lead to the solution of the Kosniowski's conjecture (Conjecture~\ref{K1979}) 
for more values of $n \in \N$, provided \eqref{assum} holds.

\section{Applications of Theorem~\ref{ab}}

The 
following result and its corollary are consequences of Theorem~\ref{ab}. They provide some 
necessary conditions on the number of negative weights at the fixed points of the action.

\begin{theoremL}\label{est1}
Let $(M,\omega)$ be a $2n$-dimensional compact and connected symplectic manifold
with a symplectic but non-Hamiltonian $S^1$\--action. Suppose that the
fixed point set of the action is nonempty and discrete, and that
{\rm (\ref{assum})} holds. Then the following hold:
\begin{itemize}
\item If $n=2m$ and $m \notin \{6k^2\mid k\in \Z \}$, then:
\begin{eqnarray} \label{eq1}
N_m+ \sum_{k=1}^{\ell} N_{m-k}>0,\quad\mbox{and}\quad \sum_{k=\ell+1} ^{m-1}N_{m-k}>0
\end{eqnarray}
where $$\ell=[\sqrt{m/6}]\,.$$\\
\item If $n=2m+1$ and $m \notin \{6k(k+1)+1\mid k\in \Z_{\geq 0}\}$, then:
\begin{eqnarray} \label{eq3}
\sum_{k=0}^{\ell} N_{m-k}>0,\quad\mbox{and}\quad \sum_{k=\ell} ^{m-1}N_{m-k}>0
\end{eqnarray}
where $$\ell=\left[\displaystyle\frac{-3+\sqrt{6m+3}}{6}\right].$$
\end{itemize}
\end{theoremL}

\begin{proof}
We may write the condition $G_1=0$ as
\begin{eqnarray} \label{g=0}
2  \sum_{k=[\sqrt{m/6}]+1}^{m-1} (6k^2-m)N_{m-k}=
mN_m+ 2 \sum_{k=1}^{[\sqrt{m/6}]}  (m-6k^2) N_{m-k}.\nonumber
\end{eqnarray}
Then the coefficients of the $N_i$'s in the sums above
are strictly positive in the range over which they are added. Since
$$
\sum_{i=1}^{n-1} N_i >0,
$$
formula (\ref{eq1}) follows.

The proof of (\ref{eq3}) is similar to the proof of (\ref{eq1}) in view of 
the fact that we may write the condition $G_2=0$  as
\begin{eqnarray}\label{g2=0}
\sum_{k=0}^{[\frac{-3+\sqrt{6m+3}}{6}]} \Big(m-1-6k(k+1)\Big)N_{m-k}  = \nonumber \\
=\sum_{k=[\frac{-3+\sqrt{6m+3}}{6}]+1} ^{m-1}\Big(6k(k+1)-m+1\Big)N_{m-k} ,\nonumber
\end{eqnarray}
and each of the coefficients multiplying the $N_i$'s is positive.
\end{proof}
From Theorem \ref{est1} we obtain the  following (which  
complements Theorem~\ref{maintheorem}).

\begin{cor}\label{bound1}
Let $(M,\omega)$ be a $2n$-dimensional compact and connected symplectic manifold
with a symplectic but non-Hamiltonian $S^1$\--action. Suppose that the
fixed point set of the action is nonempty and discrete, and that
{\rm (\ref{assum})} holds. Then the following hold:
\begin{itemize}
\item[(1)] 
Let $n=2m$ and $m \notin \{6k^2\mid k\in \Z\}$. Then:
\begin{itemize}
\item[(a)]
if $N_m>0$ or there exists a $[\sqrt{m/6}]+1 \leq k \leq m-1$ such that $N_{m-k} \neq 0$ then the minimal number of fixed points is $3$;
\item[(b)]
if there exists a $1 \leq k \leq [\sqrt{m/6}]$ such that $N_{m-k} \neq 0$
then the minimal number of fixed points
is $4$;
\item[(c)]
 if $m<6$ then 
$N_m>0$.
\end{itemize}

\item[(2)] Let $n=2m+1$ and $m \notin \{6k(k+1)+1\mid k\in \Z_{\geq 0}\}$. Then the minimal number of fixed points is $4$.
\end{itemize}
\end{cor}

\begin{proof}
(1) By \eqref{eq1},
if $N_m>0$ then  there exists a $[\sqrt{m/6}]+1 \leq k \leq m-1$ such that
$N_{m-k}\neq 0$; viceversa, if there exists a $[\sqrt{m/6}]+1 \leq k \leq m-1$ such that
$N_{m-k}\neq 0$ then either $N_m$ or $N_{m-h}>0$, where $1\leq h \leq [\sqrt{m/6}]$, and (a) follows
from Theorem~\ref{peto}.

If there exists a $1 \leq k \leq [\sqrt{m/6}]$ such that $N_{m-k} \neq 0$
then there exists $[\sqrt{m/6}]+1 \leq h \leq m-1$ such that
$N_{m-h}\neq 0$ which, by Theorem~\ref{peto}, implies (b).

The proof of (2) is analogous. 
\end{proof}

\begin{remark}
Equation {\rm (\ref{xc})} may not hold, as it may be seen from Table \ref{abd} at $n=8$. 
When $\dim(M)=6$ the equation given by $G_2=0$ is an identity thus, by Theorem \ref{peto}, $\mathcal{B}(3)=2$.
More generally,
it's easy to see that when $\dim(M)=2n=4m$ with $m\in \{6k^2\mid k\in \Z_{>0}\}$, or
when $\dim(M)=2n=4m+2$, with $m\in \{6k(k+1)+1\mid k\in \Z_{\geq 0}\}$ then, by our procedure and Theorem \ref{peto}, we get
$\mathcal{B}(n)=2$.
\end{remark}

The following is a consequence of Example~\ref{ex1}.

\begin{cor} \label{general2}
Let $(M,\omega)$ be a $2n$-dimensional compact and connected symplectic manifold
with a symplectic but non-Hamiltonian $S^1$ action with 
nonempty, discrete fixed point set and 
such that {\rm (\ref{assum})} holds. Suppose that $2n \in \{8,10,14,20\}$. Then the $S^1$\--action has at least $n + 1$ 
fixed points.
\end{cor}

Under some assumptions, we can answer the following question.

\begin{question}(\cite{PeTo2010}) \label{new}
\emph{Suppose that $n$ is an odd number.
Is there a symplectic 
$S^1$\--action on a compact, connected symplectic $2n$\--manifold $(M,\omega)$ with
exactly {\em three} fixed points, other than the standard
actions on $\mathbb{C} P^2$?}
\end{question}

Question \ref{new} was settled by Jang recently.

\begin{theorem}[\cite{Ja12}] \label{ja}
Let $S^1$ act symplectically on a compact, connected symplectic manifold $(M,\omega)$. 
If there are exactly three fixed points, $M$ is equivariantly symplectomorphic to $\mathbb{C} P^2$.
\end{theorem}

Corollary \ref{bound1} (2) gives Theorem \ref{ja} with a simpler proof in the following cases.

\begin{cor}
The answer to Question \ref{new} is ``No" whenever $\dim(M)=2m+1$, \eqref{assum} holds and $m\notin \{6k(k+1)+1 \mid  k\in \Z_{\geq 0}\}$.
\end{cor}

\begin{theoremL} \label{g1}
Let $(M,\omega)$ be an $8$-dimensional compact and connected symplectic manifold
with a symplectic but non-Hamiltonian $S^1$\--action with 
nonempty, discrete fixed point set and 
such that ${\rm c}_1^M=0$. Then $${\rm c}_2^M\neq0$$ and $$\int_M({\rm c}_2^M)^2\geq 2.$$
\end{theoremL}

\begin{proof}
The Todd genus of an $8$-dimensional compact and connected symplectic manifold is given by
$$
{\rm Todd}(M)=\int_{M}\frac{-({\rm c}_1^M)^4+4({\rm c}_1^M)^2{\rm c}_2^M+3({\rm c}_2^M)^2+{\rm c}_1^M{\rm c}_3^M-{\rm c}_4^M}{720}.
$$
Since by assumption ${\rm c}_1^M=0$ and the action is not Hamiltonian, by Theorem \ref{fel} we have
\begin{equation}\label{c2}
\int_M ({\rm c}_2^M)^2=\frac{1}{3}\int_{M}{\rm c}_4^M\;.
\end{equation}
By the Atiyah-Bott-Berline-Vergne Localization Theorem (\cite{AB,BV}), it is straightforward to see that
$$
\int_M {\rm c}_4^M= \mbox{number of fixed points of the action}.
$$ 
Since we are assuming ${\rm c}_1^M=0$, condition \eqref{assum} is satisfied. Hence, by Example \ref{ex1}, 
the number of fixed points 
is greater or equal to $6$, which, together with \eqref{c2} gives the desired inequality, and hence ${\rm c}_2^M\neq 0$.
\end{proof}

\begin{remark}
Theorem \ref{g1} does not immediately generalize to dimension 10. In this case
the Todd polynomial is of the form ${\rm c}_1^M\mathcal{C}$, where $\mathcal{C}$ is a combination
of Chern classes of degree 8. Hence, if we assume ${\rm c}_1^M=0$, the Todd genus is zero.
It would be interesting to understand under which conditions 
Theorem \ref{g1} generalizes to dimension $10$ or higher.
\end{remark}

We have the following consequence of Corollary~\ref{general} and Theorem~\ref{ab}.

\begin{theoremL}
Let $(M,\omega)$ be an $2n$-dimensional compact and connected symplectic manifold
with a symplectic but non-Hamiltonian $S^1$ action with 
nonempty, discrete fixed point set, such 
that {\rm (\ref{assum})} is satisfied. If the number of fixed points is
in $[\mathcal{B}(n),n]$ then the Chern class map is not
somewhere injective.
\end{theoremL}

\section{Final remarks} \label{fr}

Proposition~\ref{easy} follows from the fact that the momentum map $\mu$ 
is a Morse-Bott function, whose set of critical points 
$\operatorname{Crit}(\mu)$ is a submanifold of $M$, and coincides with the
fixed point set of the action. 
Thus, if it is not zero dimensional, then there are infinitely
many critical points of $\mu$ and the result
is obvious. If it is
zero dimensional, then $\mu$ is a perfect Morse function
(i.e., the Morse inequalities are equalities) because of the
following classical result: \emph{If 
$f$ is a Morse function on a compact and connected
manifold whose critical points  have only even 
indices, then it is a perfect Morse function} 
(e.g., \cite[Corollary 2.19 on page 52]{Nicolaescu2007}).

Let $m_k(\mu)$ be the number of critical points of $\mu$ of index $k$. The total number
of critical points of $\mu$ is
$$\sum_{k=0}^{2n}m_k(\mu) = 
\sum_{k=0}^{2n}{\rm b}_{k}(M),$$ where 
${\rm b}_k(M):=\dim\left({\rm H}^{k}(M, \mathbb{R}) \right)$ is the $k$th Betti number of $M$. 
The classes $[\omega^k]$ are nontrivial in
${\rm H}^{2k}(M, \mathbb{R})$ for $k=0, \ldots, n$,
so ${\rm b}_{2k}(M) \geq 1$,
and hence the number of 
critical points of $\mu$ is at least $n+1$.

One can try to use Theorem \ref{main} below to
deduce a result analogous to Proposition~\ref{easy} for circle valued 
momentum maps by replacing the Morse inequalities by the
Novikov inequalities (see \cite[Chapter 11, Proposition 2.4]{Pajitnov2006}, 
\cite[Theorem 2.4]{Farber2004}), if all the critical
points of $\mu$ are non-degenerate.

\begin{theorem}[McDuff, '88] \label{main}
Let the circle $S^1$ act 
symplectically on the compact connected symplectic 
manifold $(M, \sigma)$. 
Then either the action admits a standard 
momentum map or, if not, there
exists a $S^1$-invariant 
symplectic form $\omega$ on $M$   that 
admits a circle valued momentum map 
$\mu: M \rightarrow S^1$. 
Moreover, $\mu$ is a Morse-Bott-Novikov function 
and each connected component of 
$M^{S^1} = \operatorname{Crit}(\mu)$ has even index. 
If $ \sigma$  is
integral, then $ \omega= \sigma$. 
\end{theorem}

The number
of critical points of the circle-valued momentum
map $\mu$  in Theorem~\ref{main} is $\sum_{k=0}^{2n} m_k(\mu)$. This integer is estimated from below by
$$
\sum_{k=0}^{2n} \Big(\hat{{\rm b}}_k(M)+
\hat{{\rm q}}_k(M)+\hat{{\rm q}}_{k-1}(M) \Big),
$$
where  $\hat{{\rm b}}_k(M)$ is the rank of the 
$\mathbb{Z}((t))$-module 
${\rm H}_k(\widetilde{M},\mathbb{Z}) 
\otimes_{\mathbb{Z}[t, t^{-1}]}\mathbb{Z}((t))$, 
$\hat{{\rm q}}_k(M)$ is the torsion number of this 
module, and $\widetilde{M}$ is the pull back by 
$\mu: M \rightarrow \mathbb{R}/\mathbb{Z}$ of the 
principal $\mathbb{Z}$-bundle $t \in \mathbb{R} 
\mapsto [t] \in \mathbb{R}/\mathbb{Z}$. 
Unfortunately, this lower bound can be zero.
We refer to \cite[Sections 3 and 4]{PR12} for a detailed proof of
Theorem~\ref{main} and \cite[Remark 6]{PR12} for further details.
\smallskip
\medskip

{\emph{Acknowledgements}.  
This paper was written at the Bernoulli Center
in Lausanne (EFPL) during the program on Semiclassical Analysis and Integrable
Systems organized by \'Alvaro Pelayo, Nicolai Reshetikhin, and   San V\~u Ng\d oc, during
July 1-December 31, 2013.    We would like to thank D. McDuff for useful comments}.
 AP was partially
supported by an 
NSF CAREER DMS-1055897.

{\small
\noindent
\\
\'Alvaro Pelayo \\
School of Mathematics\\
Institute for Advanced Study\\
Einstein Drive\\
 Princeton, NJ 08540 USA.
\\
\\
\noindent
Washington University\\
Mathematics Department \\
One Brookings Drive, Campus Box 1146\\
St Louis, MO 63130-4899, USA.\\
{\em E\--mail}: \texttt{apelayo@math.wustl.edu}

\medskip\noindent

\smallskip\noindent
Silvia Sabatini\\
Section de Math\'ematiques, Station 8\\
Ecole Polytechnique F\'ed\'erale de Lausanne\\
CH-1015 Lausanne, Switzerland\\
{\em E\--mail}: \texttt{silvia.sabatini@epfl.ch}


\begin{thebibliography}{9999}

\bibitem[AtBo84]{AB} Atiyah, M.F. and R. Bott, The moment map and equivariant cohomology, \emph{Topology} {\bf 23}
(1984), 1-28.

\bibitem[Au91]{Au1991} M. Audin, The Topology of torus actions on symplectic
manifolds, Progr. in Math. 93, 1991.

\bibitem[BV82]{BV} Berline, N. and M. Vergne, Classes caracteristiques equivariantes, formule de localisation
en cohomologie equivariante, \emph{C.R. Acad. Sci. Paris} {\bf 295} (1982) 539-541.

\bibitem[CHS10]{CHS10} Cho, Y., Hwang, T., and Suh, D. Y., Semifree $S^1$\--actions on $6$\--dimensional
symplectic manifolds whose reduced spaces admit $S^2$\--fibrations, \emph{arXiv:1005.0193}.

\bibitem[DP07]{DP07} Duistermaat, J.J., and Pelayo, A., Symplectic torus
     actions with coisotropic principal orbits, {\it Annales de
       l'Institut Fourier} {\bf 57} (2007) 2239\--2327.


\bibitem[Fa04]{Farber2004}
Farber, M., \textit{Topology of Closed One-Forms},
Mathematical Surveys and Monographs, \textbf{108},
American Mathematical Society, 2004.


\bibitem[Fe01]{Feldman2001}
Feldman, K.E., Hirzebruch genera of manifolds supporting a Hamiltonian circle action (Russian),
 \textit{Uspekhi Mat. Nauk}  \textbf{56}(5) (2001), 187--188; translation in \textit{Russian Math. Surveys}
 \textbf{56} (5) (2001), 978\--979.

\bibitem[Fr59]{Frankel1959}
Frankel, T., Fixed points and torsion on K\"ahler manifolds, \textit{Ann. Math.} \textbf{70}(1) (1959), 1--8.


\bibitem[Go05]{go} Godinho, L., On certain symplectic circle actions.
{\em J. Symplectic Geom.} 3 (2005) 357\--383.



\bibitem[Go06]{Go2006}  Godinho, L., Semifree symplectic circle actions on 4-orbifolds.
\emph{Trans. Amer. Math. Soc.}  {\bf 358}  (2006)  4919\--4933.

\bibitem[GoSa12]{GoSa12} Godinho L. and Sabatini S., New tools for classifying Hamiltonian circle
actions with isolated fixed points, arxiv:1206.3195.


\bibitem[GT84]{GuSt1984} 
Guillemin, V. and Sternberg, S., 
\emph{Symplectic Techniques in Physics}. 
Cambridge University Press, Cambridge, 1984.



\bibitem[Ha83]{Ha} Hattori, A.: $S^1$\--actions with only isolated fixed points on almost complex
manifolds, Proc. Japan Acad. Ser. A Math. Sci. {\bf 59} (1983) 293\--296.



\bibitem[Ko79]{Ko}  Kosniowski, C., Some formulae and conjectures associated to circle actions,
Topology Symposium, Siegen 1979 (Prof. Symps., Univ. Siegen, 1979(, pp 331\--339. Lecture
Notes in Math 788, Springer, Berlin 1980. 

\bibitem[Ja12]{Ja12} Jang, D., Symplectic periodic flows with exactly three equilibrium points.
http://arxiv.org/abs/1210.0458v1.


 \bibitem[LL10]{LiLi2010}  Li, P. and Liu, K., Some remarks on circle action on manifolds
 \emph{Math. Res. Letters}. {\bf 18} (2011) 435\--446.



\bibitem[LL11]{LL11} Li, P., and Liu, K., On an algebraic formula and applications to group
actions on manifolds, arxiv:11106.0147.


\bibitem[Li07]{Lin2007} Lin, Y., The log-concavity conjecture for the Duistermaat-Heckman measure revisited,
\emph{International Mathematics  Research  Notices IMRN  (2008)} Vol. 2008, article ID rnn027, 19 pages,
doi:10.1093/imrn/rnn027.

\bibitem[LP12]{LP12} Lin, Y., and Pelayo, \'A.,   Log-concavity and symplectic flows, arXiv:1207.1335. 


\bibitem[Mc88]{McDuff1988}
McDuff, D., The moment map for circle actions on symplectic manifolds, \textit{J. Geom. Phys.}, \textbf{5}(2) (1988), 149--160.

\bibitem[MS98]{McSa1998} 
McDuff, D. and Salamon, D., 
\emph{Introduction to Symplectic Topology}, second 
edition, Oxford Mathematical Monographs, The Clarendon Press, Oxford University Press, New York, 1998. 

\bibitem[Mi63]{Milnor1963}
Milnor, J., \textit{Morse Theory}, Annals of Mathematics Studies, \textbf{51}, Princeton University Press, Princeton, N.J., 1963.


\bibitem[Ni07]{Nicolaescu2007}
Nicolaescu, L., \textit{An Invitation to Morse
Theory}, Universitext, Springer-Verlag, New York,
2007.


\bibitem[Pa06]{Pajitnov2006}
Pajitnov, A., \textit{Circle-valued Morse Theory}, 
de Guyter Studies in Mathematics, \textbf{32},
Walter de Gruyter, Berlin, New York, 2006.

\bibitem[Pe10]{Pe10} Pelayo, \'A., Symplectic actions of $2$\--tori on
     $4$\--manifolds, \emph{Mem. Amer. Math. Soc.} {\bf 204} (2010)
     no. 959.


\bibitem[PR12]{PR12} Pelayo, \'A., Ratiu, T. S., Circle\--valued momentum
maps for symplectic periodic flows. \emph{Enseign. Math. (2)} {\bf 58}
(2012) 205\--219.

\bibitem[PT11]{PeTo2010}
Pelayo, A. and Tolman, S., Fixed points of
symplectic periodic flows, \emph{Ergod. Theory and Dyn. Syst.} {\bf 31} 
(2011) 1237\--1247.

\bibitem[T10]{T10} Tolman, S., On a symplectic generalization of Petrie's conjecture. \emph{Trans. Amer. Math.
Soc.} {\bf 362} (2010), 3963-3996.

\bibitem[TW00]{tolmanweitsman} Tolman, S. and Weitsman, J., On semi\--free
symplectic circle actions with isolated fixed points. {\em Topology}
{\bf 39} (2000) 299-309.


\end{thebibliography}
 \end{document}